\documentclass[twocolumns]{autart}

\usepackage{setspace}

\usepackage{url}
\usepackage[numbers]{natbib}            
\usepackage{graphicx}          
\usepackage{slashbox}                               

\usepackage{amssymb}
\usepackage{amsmath}
\usepackage{graphicx}
\usepackage{amsfonts}
\usepackage{multirow}
\usepackage{color}
\usepackage{psfrag}
\usepackage{mathrsfs}
\usepackage{dsfont}
\usepackage{etoolbox}

 \usepackage{tabu}

\selectfont

\renewcommand{\theenumi}{\alph{enumi}}

\everymath{\displaystyle}

\newtheorem{theorem}{Theorem}[section]
\newtheorem{corollary}[theorem]{Corollary}
\newtheorem{proposition}[theorem]{Proposition}

\newenvironment{proof}{{\it Proof :~}}{\hfill$\diamondsuit$\\}

\newtheorem{example}{Example}

\DeclareMathOperator*{\col}{col}

\DeclareMathOperator{\eps}{\varepsilon}

\DeclareMathOperator{\co}{\mathbf{co}}

\begin{document}


\begin{frontmatter}

\title{Convex lifted conditions for robust stability analysis and stabilization of linear discrete-time switched systems}

\author{Corentin Briat}\ead{briatc@bsse.ethz.ch,corentin@briat.info}\ead[url]{http://www.briat.info}

\address{Swiss Federal Institute of Technology--Z\"{u}rich (ETH-Z), Department of Biosystems Science and Engineering (D-BSSE), Switzerland.}

%


\begin{abstract}
Stability analysis of discrete-time switched systems under minimum dwell-time is studied using a new type of LMI conditions. These conditions are convex in the matrices of the system and shown to be equivalent to the nonconvex conditions proposed by Geromel and Colaneri in \citep{Geromel:06d}. The convexification of the conditions is performed by a lifting process which introduces a moderate number of additional decision variables. The convexity of the conditions can be exploited to extend the results to uncertain systems, control design and $\ell_2$-gain computation without introducing additional conservatism. Several examples are presented to show the effectiveness of the approach.
\end{abstract}

\end{frontmatter}

\section{Introduction}

Switched systems have been shown to provide a general framework for modeling real-world systems such as time-delay systems \citep{Hetel:08b}, networked control systems \citep{Donkers:11}, biomedical problems \citep{Hernandez:11}, etc. Both continuous-time \citep{Morse:96,Hespanha:99,Goebel:09} and discrete-time instances \citep{Daafouz:02,Hetel:07,Lin:08,Colaneri:11,Geromel:06d} of these systems have been theoretically studied in the literature over the past decades.  Due to their time-varying nature, they may indeed exhibit very interesting and intricate behaviors. For instance, switching between stable subsystems may not result in an overall stable switched system whereas switching between unstable subsystems may give rise to asymptotically stable trajectories; see e.g. \citep{Decarlo:00,Liberzon:99}.

A way for analyzing stability of switching systems is through the notion of dwell-time: minimum \citep{Morse:96,Geromel:06b,Geromel:06d} and average dwell-times \citep{Hespanha:99,Zhang:08,Zhang:09} are the most usual ones. It has been shown quite recently that stability under minimum dwell-time can be analyzed in a simple way for both continuous- and discrete-time systems \citep{Geromel:06b,Geromel:06d,Chesi:12}. For linear systems, the conditions obtained using quadratic Lyapunov functions are expressed in terms of LMIs which may sometimes yield  tight results, even though these conditions are not necessary in general. Necessity can be recovered by considering more general Lyapunov functions, such as homogeneous ones \cite{Chesi:12}. The obtained stability conditions are, however, nonconvex in the system matrices and are, therefore, difficult to extend  to uncertain systems and to control design. 

In this paper, the minimum dwell-time conditions of \citep{Geromel:06d} are considered back and reformulated in a `lifting setting', which has been recently considered in \citep{Briat:13d}, in the continuous-time setting, in order to overcome certain limitations in control design arising in the use of certain functionals, such as looped-functionals; see e.g. \citep{Seuret:12,Briat:11l,Briat:12d,Briat:13b}. The lifted conditions, taking the form of a sequence of inter-dependent LMIs, are shown to be equivalent to the conditions of \citep{Geromel:06d}. Despite being equivalent, they are convex in the matrices of the system, a key property for considering uncertain systems and for obtaining tractable design conditions. The approach is finally extended to the problem of computation of the $\ell_2$-gain of discrete-time switched systems under dwell-time constraint. Some remarks on the associated stabilization problem are provided as well. Several numerical examples are considered in order to emphasize the effectiveness of the approach.

\vspace{-2mm}\textbf{Outline:} Preliminaries are given in Section \ref{sec:prel}. Section \ref{sec:affine} is devoted to stability analysis of switched systems using novel convex conditions. These results are extended in Section \ref{sec:rob} to the case of uncertain systems whereas Section \ref{sec:stabz} pertains on stabilization. Section \ref{sec:ext} finally addresses the computation of an upper-bound on the $\ell_2$-gain under dwell-time constraint. Examples are treated in the related sections.

\vspace{-2mm}\textbf{Notations:} The set of $n\times n$ (positive definite) symmetric matrices is denoted by ($\mathbb{S}_{\succ0}^n$) $\mathbb{S}^n$. Given two symmetric matrices $A,B$, the inequality $A\succ(\succeq) B$ means that $A-B$ is positive (semi)definite. The transpose of the matrix $A$ is denoted by $A^\prime$. The $\ell_2$-norm of the sequence $w:\mathbb{N}\to\mathbb{R}^n$ is denoted by $\textstyle||w||_{\ell_2}:=\left(\sum_{k=0}^\infty w(k)^\prime w(k)\right)^{1/2}$.

\section{Preliminaries}\label{sec:prel}

Let us consider the following class of linear switched systems
\begin{equation}\label{eq:mainsyst}
\begin{array}{rcl}
  x(t+1)&=&A_{\sigma(t)}x(t)\\
  x(t_0)&=&x_0
\end{array}
\end{equation}
where $x,x_0\in\mathbb{R}^n$ are the state of the system and the initial condition, respectively. The switching signal $\sigma$ is defined as
\begin{equation}
  \sigma:\mathbb{N}\to\{1,\ldots,N\}
\end{equation}
where $N$ is the number of subsystems involved in the switched system. Let the sequence $\{\phi_q\}_{q\in\mathbb{N}}$ be the sequence of switching instants, i.e. the instants where $\sigma(t)$ changes value and let $\tau_q:=\phi_{q+1}-\phi_q$ be the so-called dwell-time. By convention, we set $\phi_0=0$.


We recall now several existing results on stability of discrete-time switched systems. The following one has been derived in \citep{Geromel:06d}:
\begin{theorem}\label{th:geromel}
  Assume that, for some $\tau>0$, there exist matrices $P_i\in\mathbb{S}^n_{\succ0}$, $i=1,\ldots,N$, such that the LMIs
  \begin{equation}
    A_i^\prime P_iA_i-P_i\prec0
  \end{equation}
  and
  \begin{equation}\label{eq:geromel2}
    A_i^{\prime\tau}P_j A_i^\tau-P_i\prec0
  \end{equation}
  hold for all $i,j=1,\ldots,N$, $i\ne j$. Then, the switched system \eqref{eq:mainsyst} is asymptotically stable for all switching-time sequences $\{\phi_q\}_{q\in\mathbb{N}}$ satisfying $\tau_q\ge \tau$.
\end{theorem}
When $\tau=1$ we get, as a corollary, the following result initially proved in \citep{Daafouz:02}:
\begin{corollary}\label{th:daafouz}
   Assume that there exist matrices $P_i\in\mathbb{S}^n_{\succ0}$, $i=1,\ldots,N$, such that the LMIs
  \begin{equation}
    A_i^\prime P_jA_i-P_i\prec0
  \end{equation}
   hold for all $i,j=1,\ldots,N$. Then, the switched system \eqref{eq:mainsyst} is asymptotically stable for any switching-time sequence $\{\phi_q\}_{q\in\mathbb{N}}$.
\end{corollary}

The following result is equivalent to Theorem \ref{th:geromel}:
\begin{theorem}\label{th:geromel2}
  Assume that, for some $\tau>0$, there exist matrices $P_i\in\mathbb{S}^n_{\succ0}$, $i=1,\ldots,N$, such that the LMIs
  \begin{equation}
    A_i^\prime P_iA_i-P_i\prec0
  \end{equation}
  and
  \begin{equation}
    A_i^{\prime\tau}P_i A_i^\tau-P_j\prec0
  \end{equation}
  hold for all $i,j=1,\ldots,N$, $i\ne j$. Then, the switched system \eqref{eq:mainsyst} is asymptotically stable for all switching-time sequences $\{\phi_q\}_{q\in\mathbb{N}}$ satisfying $\tau_q\ge \tau$.
\end{theorem}
The main difference between Theorem \ref{th:geromel} and Theorem \ref{th:geromel2} lies in the second LMI condition, where the matrices $P_i$ and $P_j$ have been swapped. Theorem \ref{th:geromel} indeed considers the Lyapunov function given by
\begin{equation}
  V_f(x(t),\sigma(t))=x(t)^\prime P_{\sigma(t)}x(t)
\end{equation}
whereas Theorem  \ref{th:geromel2} considers the Lyapunov function
\begin{equation}
  V_b(x(t),\sigma(t-1))=x(t)^\prime P_{\sigma(t-1)}x(t).
\end{equation}
The difference between these two Lyapunov functions is illustrated in Figure~\ref{fig:LF} where we can see that the trajectories slightly differ but are both monotonically decreasing.
\begin{figure}
  \centering
  \includegraphics[width=0.5\textwidth]{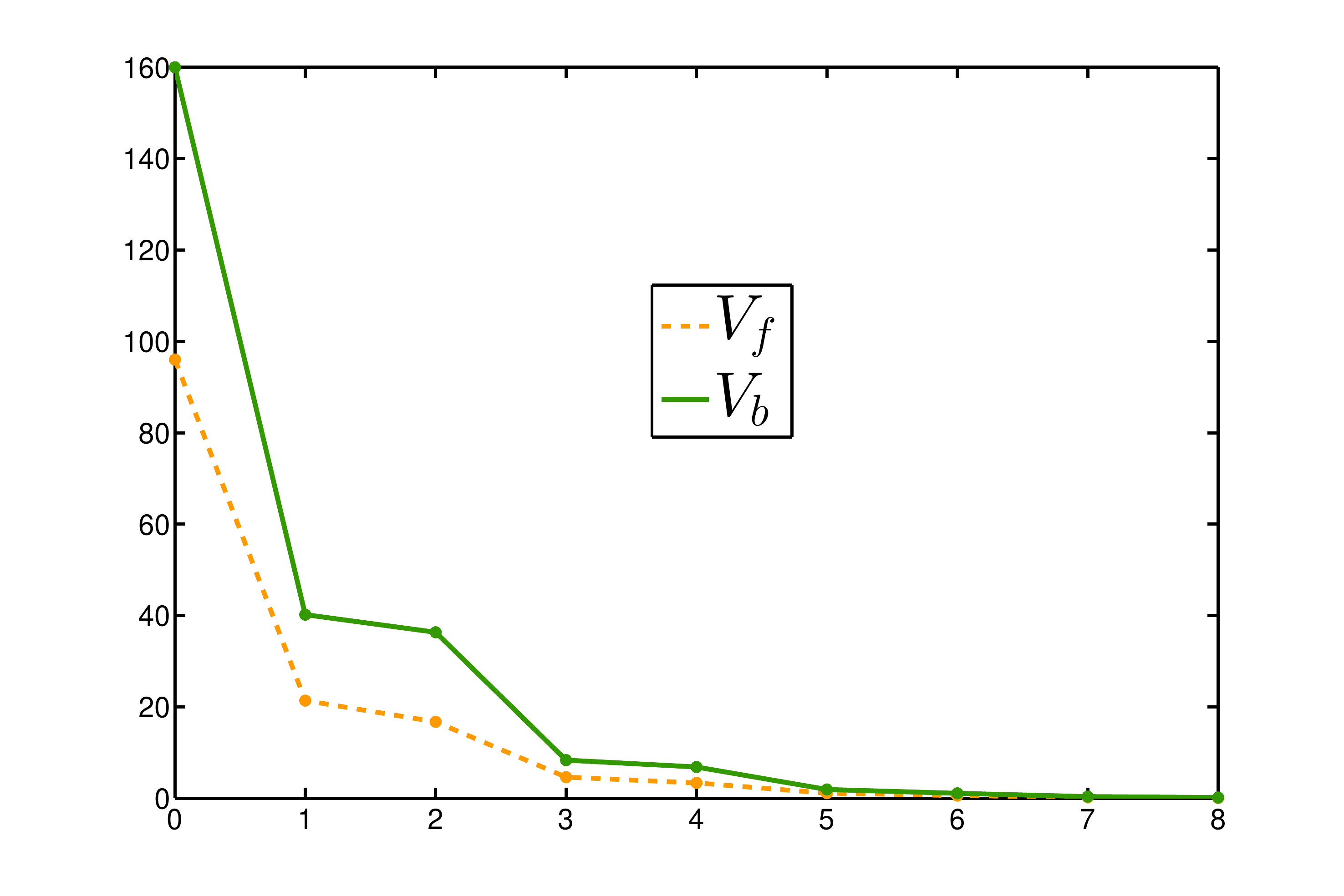}
  \caption{Evolution of the Lyapunov functions $V_f$ and $V_b$ along the trajectories of a discrete-time switched system of the form \eqref{eq:mainsyst}.}\label{fig:LF}
\end{figure}

%
%

\section{Convex conditions for minimum dwell-time analysis}\label{sec:affine}

The conditions stated in Theorem \ref{th:geromel} and Theorem \ref{th:geromel2} are nonconvex in the matrices of the system due to the presence of powers of these matrices. Equivalent convex conditions are proposed in this section.

\subsection{Main results}

The following result is the convex counterpart of Theorem \ref{th:geromel2}:
\begin{theorem}\label{th:affine}
 The following statements are equivalent:
 \begin{enumerate}
       %
   \item There exist matrices $P_i\in\mathbb{S}_{\succ0}^n$, $i=1,\ldots,N$ such that the LMIs
  \begin{equation}\label{eq:cond1}
    A_i^\prime P_iA_i-P_i\prec0
  \end{equation}
  and
  \begin{equation}\label{eq:cond2}
    A_i^{\prime\tau}P_i A_i^\tau-P_j\prec0
  \end{equation}
  hold for all $i,j=1,\ldots,N$, $i\ne j$.
   \item There exist matrix sequences $R_i:\{0,\ldots,\tau\}\to\mathbb{S}^n$, $R_i(0)\succ0$, $i=1,\ldots,N$, and a scalar $\eps>0$ such that the LMIs
   \begin{equation}\label{eq:cond1a}
    A_i^{\prime}R_i(\tau)A_i-R_i(\tau)\prec0
   \end{equation}
    \begin{equation}\label{eq:cond2a}
    A_i^\prime R_i(k+1)A_i-R_i(k)\preceq0
  \end{equation}
  and
  \begin{equation}\label{eq:cond2b}
    R_i(0)-R_j(\tau)+\eps I\preceq0
  \end{equation}
  hold for all $i,j=1,\ldots,N$, $i\ne j$ and $k=0,\ldots,\tau-1$.
     \item There exist matrix sequences $S_i:\{0,\ldots,\tau\}\to\mathbb{S}^n$, $S_i(\tau)\succ0$, $i=1,\ldots,N$, and a scalar $\eps>0$ such that the LMIs
   \begin{equation}\label{eq:cond1c}
    A_iS_i(\tau)A_i^{\prime}-S_i(\tau)\prec0
   \end{equation}
    \begin{equation}\label{eq:cond3a}
    A_iS_i(k)A_i^{\prime}-S_i(k+1)\preceq0
  \end{equation}
  and
  \begin{equation}\label{eq:cond3b}
    S_j(\tau)-S_i(0)+\eps I\preceq0
  \end{equation}
  hold for all $i,j=1,\ldots,N$, $i\ne j$ and $k=0,\ldots,\tau-1$.
   \end{enumerate}
   Moreover, when one of the above equivalent statements holds, then the switched system \eqref{eq:mainsyst} is asymptotically stable for all switching-time sequences $\{\phi_q\}_{q\in\mathbb{N}}$ satisfying $\tau_q\ge \tau$.
  \end{theorem}
  \begin{proof}
  %
    \textbf{Proof of (b) $\Rightarrow$ (a):} We have to prove here that conditions \eqref{eq:cond2a} and \eqref{eq:cond2b} together imply that condition \eqref{eq:cond2} holds. Let us denote $L_i(k):=A_i^\prime R_i(k+1)A_i-R_i(k)\preceq0$ and consider the sum
    \begin{equation}
      \sum_{k=0}^{\tau-1} A_i^{\prime k}L_i(k)A_i^k=A_i^{\prime\tau}R_i(\tau) A_i^\tau-R_i(0)\preceq0.
    \end{equation}
    Using then condition \eqref{eq:cond2b}, we get that
    \begin{equation}
    A_i^{\prime\tau}R_i(\tau)A_i^\tau-R_j(\tau)+\eps I\preceq0
  \end{equation}
  which implies in turn that \eqref{eq:cond2} holds with $P_i=R_i(\tau)$.

  \textbf{Proof of (a) $\Rightarrow$ (b):} To prove this, we first show that \eqref{eq:cond2a} always has solutions, regardless of the stability of the system. Then, we combine the solution to \eqref{eq:cond2} to show that \eqref{eq:cond2b} holds.

  Solving then for
  \begin{equation}\label{eq:glabglab}
    A_i^\prime R_i(k+1)A_i-R_i(k)=-W_i(k)
  \end{equation}
  for some $W_i(k)\succeq0$,  we get that
  \begin{equation}\label{eq:rik}
    R_i(k)=A_i^{\prime (\tau-k)}R_i(\tau)A_i^{\tau-k}-\bar{W}_i(k)
  \end{equation}
  where $\textstyle\bar{W}_i(k):=\sum_{\iota=k}^{\tau-1}A_i^{(\iota-k)\prime}W_i(\iota)A_i^{\iota-k}$ and hence
  \begin{equation}
    R_i(0)=A_i^{\prime \tau}R_i(\tau)A_i^{\tau}-\bar{W}_i(0).
  \end{equation}
  Substitute now $R_i(0)$ in \eqref{eq:cond2} with $P_i=R_i(\tau)$ to get
  \begin{equation}
     R_i(0)-R_j(\tau)\prec-\bar{W}_i(0)
  \end{equation}
  which is equivalent to \eqref{eq:cond2b} since $\bar{W}_i(0)\succeq0$. The proof is complete.

  \textbf{Proof of (b) $\Leftrightarrow$ (c):} The proof follows from Schur complements.
  \end{proof}

The rationale for developing the conditions of statements (b) and (c) is to get rid of powers of the matrices of the system, i.e. $A^\tau$, that are responsible for the lack of convexity of the conditions. Unlike the conditions of statement (a), the conditions in statements (b) and (c) are convex in the matrices of the system as shown below:
\begin{proposition}\label{prop:convex}
  The LMIs \eqref{eq:cond1a}-\eqref{eq:cond2a} are convex in the $A_i$'s.
\end{proposition}
\begin{proof}
  To prove this, it is enough to prove that the matrices $R_i(k)$ are positive definite. By assumption, $R_i(0)$ is positive definite, hence the LMI \eqref{eq:cond1a} is convex in the system matrices. To see that all the $R_i$'s are positive definite, let us first notice that if the LMIs $R_i(0)\succ0$ and \eqref{eq:cond2b} hold, then we have that $R_i(\tau)\succeq R_i(0)+\eps I\succ R_i(0)$ and hence $R_i(\tau)$ is positive definite as well. Using now the equality \eqref{eq:rik}, and using the fact that $\bar{W}_i(k)\succeq0$ and that $A_i^{\prime \tau}R_i(\tau)A_i^{\tau}\prec R_i(\tau)$ from \eqref{eq:cond1a}, we get that $R_i(k)$ is decreasing as $k$ increases. We proved above that this sequence is lower bounded by $R_i(0)\succ0$, therefore $R_i(k)\succ0$ for all $k=0,\ldots,\tau-1$. The proof is complete.
\end{proof}

\subsection{Discussion}

The price to pay for this convexity, however, is the increase of the computational complexity. As shown in Table \ref{tab:comp}, the computational complexity of the lifted-conditions is higher, and affine in the dwell-time value $\tau$. This means that the increase of the computational complexity will be reasonable whenever the minimum dwell-time is small (and when the product $Nn$ is not too large). This is a quite convenient property since in most of the applications the dwell-time is aimed to be minimized.

Note, moreover, that the computational complexity is intrinsically low since the number of decision matrices is small.  As a comparison, the LMI \eqref{eq:geromel2} could be converted into an affine form using the Finsler's lemma \cite{SkeltonIG:97a} by introducing a large number of slack-variables which would introduce extra computational complexity both in the LMI size (proportional to $\tau$) and the number of variables (proportional to $N^2\tau^2$). In this respect, the proposed approach is more suitable since more tractable. Note also that, on the top of this, the Finsler's lemma yields LMI conditions that are difficult to turn into convex design conditions due to the excessive amount of slack variables that are introduced in the process.

It is finally important to stress that the computational complexity of the method is also reduced by the fact that we only impose $R_i(0)$ to be positive definite. As shown in Proposition \ref{prop:convex}, there is indeed no need to impose that $R_i(k)\succ0$ for all $k=0,\ldots,\tau-1$.

\begin{table*}
\centering
\caption{Computational complexity of the conditions of Theorem \ref{th:affine}.}\label{tab:comp}
  {\tabulinesep=1.2mm
  \begin{tabu}{|c|c|c|}
  \hline
     & no. variables & LMI size\\
     \hline
     \hline
     Th. \ref{th:affine}, (a) & $\dfrac{Nn(n+1)}{2}$ & $N(N+1)n$\\
     \hline
     Th. \ref{th:affine}, (b) & $\dfrac{N(\tau+1)n(n+1)}{2}+1$ & $(N^2+N+N\tau)n+1$\\
     \hline
  \end{tabu}}
\end{table*}

\subsection{Examples}

Several examples are addressed in what follows. The LMI parser Yalmip \citep{YALMIP} and the semidefinite programming solver SeDuMi \citep{Sturm:01a} are considered.

\begin{example}\label{ex:1}
  Let us consider the system \eqref{eq:mainsyst} with matrices $A_i=e^{B_iT}$ where  \citep{Geromel:06d}
  \begin{equation}
    B_1=\begin{bmatrix}
      0 & 1\\
      -10 & -1
    \end{bmatrix}\ \textnormal{and}\   B_2=\begin{bmatrix}
      0 & 1\\
      -0.1 & -0.5
    \end{bmatrix}.
  \end{equation}
  We set $T=0.5$ as in \citep{Geromel:06d} and use the conditions of Theorem \ref{th:affine}, statement (b). We obtain the upper-bound on the minimum dwell-time given by $\tau^*=6$. The same value is obtained using the statement (a), which is expected since the methods are equivalent. As shown in  \citep{Geromel:06d}, this bound moreover coincides with the actual minimum dwell-time since the spectrum of $A_1^5A_2^5$ contains one eigenvalue outside the unit disc. 
\end{example}

\begin{example}
   Let us consider the system \eqref{eq:mainsyst} with matrices
  \begin{equation}
    \begin{array}{lcl}
      A_1&=&\begin{bmatrix}
        -1.3&  -0.8&   -1.5&    2.1\\
   -1.5 & -0.4 & -1.5& -0.2\\
    1.6&    0.6   & 1.8  & -2.2\\
    0.5&    0.3  &  0.5  & 0.9
      \end{bmatrix},\\
      A_2&=&\begin{bmatrix}
        -0.4&   -0.7&    0.3&    0.2\\
        -0.4&   -0.4&   -0.2&   -0.3\\
         0.6&    0.4&    0.1&    0.4\\
         0.5&    0.6&         0 &        0
      \end{bmatrix}.
    \end{array}
  \end{equation}
  Applying statement (b) of Theorem \ref{th:affine}, we find that an upper-bound on the minimal dwell-time is $\tau^*=4$. As in the previous example, this bound is the actual minimum dwell-time since one eigenvalue of the product $A_1^3A_2^3$ lies outside the unit disc.
\end{example}

\begin{example}
   Let us consider the system \eqref{eq:mainsyst} with matrices
  \begin{equation}
    \begin{array}{lclclcl}
      A_1&=&\begin{bmatrix}
        1.297  &  0.35\\
         -2.229  & -1.297
      \end{bmatrix},&&\   A_2&=&\begin{bmatrix}
      1.082&    2.67\\
   -0.079&   -1.082
      \end{bmatrix}.
    \end{array}
  \end{equation}
  These matrices have eigenvalues very close to the unit circle. It is therefore expected to have a large minimum dwell-time.  Applying statement (b) of Theorem \ref{th:affine}, we find that the upper-bound on the minimal dwell-time is $\tau^*=16$. As in the previous example, this bound is the actual minimum dwell-time since one eigenvalue of the product $A_1^{15}A_2^{15}$ lies outside the unit disc.
\end{example}

\section{Convex conditions for robust minimum dwell-time analysis}\label{sec:rob}

Let us consider now that the matrices of the system~\eqref{eq:mainsyst} are uncertain, possibly time-varying, and belonging to the following polytopes
\begin{equation}\label{eq:uncmat}
  A_i\in\mathcal{A}_i:=\co\left\{A_{i,1},\ldots,A_{i,\eta}\right\},
\end{equation}
where $\co\{\cdot\}$ is the convex-hull operator and $\eta\in\mathbb{N}$ is the number of vertices of the polytope. Let us, moreover, define the set
\begin{equation}
  \boldsymbol{\Pi}_i^\tau:=\left\{\prod_{k=1}^\tau M_k:\ M_k\in\mathcal{A}_i\right\}
\end{equation}
which contain all the possible products of $\tau$ matrices drawn from the polytope $\mathcal{A}_i$.

The following result is the robustification of Theorem \ref{th:affine}:
\begin{theorem}\label{th:affine_rob}
 The following statements are equivalent:
 \begin{enumerate}
       %
   \item There exist matrices $P_i\in\mathbb{S}_{\succ0}^n$, $i=1,\ldots,N$ such that the LMIs
  \begin{equation}\label{eq:cond1_rob}
    A_i^\prime P_iA_i-P_i\prec0
  \end{equation}
  and
  \begin{equation}\label{eq:cond2_rob}
    \Pi_i^{\prime}P_i \Pi_i-P_j\prec0
  \end{equation}
  hold for all $i,j=1,\ldots,N$, $i\ne j$, all $A_i\in\mathcal{A}_i$ and all $\Pi_i\in\boldsymbol{\Pi}_i^\tau$.
   \item There exist matrix sequences $R_i:\{0,\ldots,\tau\}\to\mathbb{S}^n$, $R_i(0)\succ0$, $i=1,\ldots,N$, and a scalar $\eps>0$ such that the LMIs
   \begin{equation}\label{eq:cond1a_rob}
    A_{i,\kappa}^{\prime}R_i(\tau)A_{i,\kappa}-R_i(\tau)\prec0
   \end{equation}
    \begin{equation}\label{eq:cond2a_rob}
    A_{i,\kappa}^\prime R_i(k+1)A_{i,\kappa}-R_i(k)\preceq0
  \end{equation}
  and
  \begin{equation}\label{eq:cond2b_rob}
    R_i(0)-R_j(\tau)+\eps I\preceq0
  \end{equation}
  hold for all $i,j=1,\ldots,N$, $i\ne j$, $k=0,\ldots,\tau-1$ and $\kappa=1,\ldots,\eta$.
     \item There exist matrix sequences $S_i:\{0,\ldots,\tau\}\to\mathbb{S}^n$, $S_i(\tau)\succ0$, $i=1,\ldots,N$, and a scalar $\eps>0$ such that the LMIs
   \begin{equation}\label{eq:cond1c_rob}
    A_{i,\kappa}S_i(\tau)A_{i,\kappa}^{\prime}-S_i(\tau)\prec0
   \end{equation}
    \begin{equation}\label{eq:cond3a_rob}
    A_{i,\kappa}S_i(k)A_{i,\kappa}^{\prime}-S_i(k+1)\preceq0
  \end{equation}
  and
  \begin{equation}\label{eq:cond3b_rob}
    S_j(\tau)-S_i(0)+\eps I\preceq0
  \end{equation}
  hold for all $i,j=1,\ldots,N$, $i\ne j$, $k=0,\ldots,\tau-1$ and $\kappa=1,\ldots,\eta$.
   \end{enumerate}

      Moreover, when one of the above equivalent statements holds, then the uncertain switched system \eqref{eq:mainsyst}-\eqref{eq:uncmat} is asymptotically stable for all switching-time sequences $\{\phi_q\}_{q\in\mathbb{N}}$ satisfying $\tau_q\ge \tau$.
  \end{theorem}
  \begin{proof}
    The proof of the results follows from simple convexity arguments.
  \end{proof}
%

\begin{example}
  Let us consider the uncertain system \eqref{eq:mainsyst}-\eqref{eq:uncmat} with polytopes
\begin{equation}
  \begin{array}{lcl}
    \mathcal{A}_1&:=&\left\{\begin{bmatrix}
             0.77&   0.88\\
            -0.58  & -0.90
    \end{bmatrix},\ \begin{bmatrix}
       0.91 &  2.23\\
      -0.01 & -0.46
    \end{bmatrix}\right\}\\
    \mathcal{A}_2&:=&\left\{\begin{bmatrix}
           0.24  & 4.42\\
           -0.10&  -1.21
    \end{bmatrix},\ \begin{bmatrix}
      0.52 &   0.49\\
   -0.08 &  -0.19
    \end{bmatrix}\right\}.
  \end{array}
  \end{equation}
 Using Theorem \ref{th:affine_rob}, we find the upper-bound on the minimum dwell-time $\tau^*=3$. It can easily be seen that this bound is nonconservative since the product $A_1(\lambda_1)A_1(\lambda_2)A_2(\lambda_3)^2$ with $A_i(\lambda)=\lambda A_{i,1}+(1-\lambda)A_{i,2}$, $\lambda_1=0.9$, $\lambda_2=0$ and $\lambda_3=1$ has an eigenvalue outside the unit disc and is therefore unstable.
\end{example}

\section{Stabilization with minimum dwell-time}\label{sec:stabz}

It is now shown that the current framework can be efficiently and accurately used for control design. To this aim, let us consider the switched system with input:
\begin{equation}\label{eq:impsystc}
    x(t+1)=A_{\sigma(t)}x(t)+B_{\sigma(t)}u_{\sigma(t)}(t)
\end{equation}
where $u_i\in\mathbb{R}^{m_i}$, $i=1,\ldots,N$ are the control input vectors with possible different dimensions.

%

We consider the following class of state-feedback control laws
\begin{equation}\label{eq:impsf}
  u_{\sigma(\phi_q)}(\phi_q^k)=\left\{\begin{array}{l}
    K_{\sigma(\phi_q)}(k)x(\phi_q^k), k\in\{0,\ldots,\tau-1\}\\
    K_{\sigma(\phi_q)}(\tau)x(\phi_q^k), k\in\{\tau,\ldots,\tau_q-1\}
  \end{array}\right.
\end{equation}
where $\phi_q^k:=\phi_q+k$. Note that when $\tau_q=\tau$, the second part of the controller is not involved.
%
%
%
%
%
%

The purpose of this section is therefore to provide tractable conditions for finding suitable matrix sequences $K_i:\{0,\ldots,\tau\}\to\mathbb{R}^{m_i\times n}$ such that the closed-loop system \eqref{eq:impsystc}-\eqref{eq:impsf} is asymptotically stable.
\begin{theorem}\label{th:stabz}
 The following statements are equivalent:
 \begin{enumerate}
       %
   \item There exist matrices $P_i\in\mathbb{S}_{\succ0}^n$ and matrix sequences $K_i:\{0,\ldots,\tau\}\to\mathbb{R}^{m_i\times n}$, $i=1,\ldots,N$, such that the LMIs
  \begin{equation}\label{eq:cond1_stabz}
    (A_i+B_iK_i(\tau))^\prime P_i(A_i+B_iK_i(\tau))-P_i\prec0
  \end{equation}
  and
  \begin{equation}\label{eq:cond2_stabz}
    \Psi_i(\tau)^{\prime}P_i \Psi_i(\tau)-P_j\prec0
  \end{equation}
  hold for all $i,j=1,\ldots,N$, $i\ne j$, where
  \begin{equation}
    \Psi_i(\tau)=\prod_{k=0}^{\tau-1} (A_i+B_iK_i(k)).
  \end{equation}
     \item There exist matrix sequences $S_i:\{0,\ldots,\tau\}\to\mathbb{S}^n$, $S_i(\tau)\succ0$, $U_i:\{0,\ldots,\tau\}\to\mathbb{R}^{m_i\times n}$, $i=1,\ldots,N$, and a scalar $\eps>0$ such that the LMIs
   \begin{equation}\label{eq:cond1c_stabz}
   \begin{bmatrix}
     -S_i(\tau) & A_iS_i(\tau)+B_iU_i(\tau)\\
     \star & -S_i(\tau)
   \end{bmatrix}\prec0
   \end{equation}
    \begin{equation}\label{eq:cond3a_stabz}
    \begin{bmatrix}
      -S_i(k+1) & A_iS_i(k)+B_iU_i(k)\\
     \star & -S_i(k)
    \end{bmatrix}\preceq0
  \end{equation}
  and
  \begin{equation}\label{eq:cond3b_stabz}
    S_j(\tau)-S_i(0)+\eps I\preceq0
  \end{equation}
  hold for all $i,j=1,\ldots,N$, $i\ne j$ and $k=0,\ldots,\tau-1$.
%
   \end{enumerate}

    Moreover, when one of the above equivalent statements holds, then the closed-loop switched system \eqref{eq:impsystc}-\eqref{eq:impsf} is asymptotically stable for all switching-time sequences $\{\phi_q\}_{q\in\mathbb{N}}$ satisfying $\tau_q\ge \tau$ with the controller gains
    \begin{equation}
        K_i(k)=U_i(k)S_i(k)^{-1}.
    \end{equation}
  \end{theorem}
\begin{proof}
\textbf{Step 1:} The first step of the proof concerns the fact that statement (a) implies that the closed-loop system is stable with minimum dwell-time $\tau$. To show this, let $\Psi_i:\mathbb{N}\to\mathbb{R}^{n\times n}$ be defined as
\begin{equation}
  \Psi_i(k+1)=\left\{\begin{array}{lcl}
                \prod_{j=0}^{k}\bar{A}_{i,j},\  k=0,\ldots,\tau-1,\\
                \bar{A}_{i,\tau}^{k+1-\tau}\Psi_i(\tau),\  k\ge\tau.
  \end{array}\right.
\end{equation}
where $\bar{A}_{i,k}=A_i+B_iK_i(k)$. Assume $k\ge\tau$, then we have
\begin{equation}
\begin{array}{lcl}
    \Psi(k)^\prime P_i \Psi_i(k)&=&\Psi_i(\tau)^{\prime}\bar{A}_{i,\tau}^{\prime{(k-\tau)}}P_i\bar{A}_{i,\tau}^{k-\tau}\Psi_i(\tau)\\
    &\prec& \Psi_i(\tau)^{\prime}P_i\Psi_i(\tau)
\end{array}
\end{equation}
where the inequality has been obtained using condition \eqref{eq:cond1_stabz}. Therefore, conditions \eqref{eq:cond1_stabz} and \eqref{eq:cond2_stabz}, together, imply that
\begin{equation}
\begin{array}{lcl}
    \Psi(k)^\prime P_i \Psi_i(k)-P_j\prec0
\end{array}
\end{equation}
for all $k\ge\tau$, proving that the system is stable with minimum dwell-time $\tau$.

\textbf{Step 2:}  The equivalence between statements (a) and (b) follows from statement (c) of Theorem \ref{th:affine}, Schur complements and the change of variables $U_i(k)=K_i(k)S_i(k)$. The proof is complete.
\end{proof}


\begin{example}\label{eq:stabz}
  Let us consider the system \eqref{eq:impsystc} with matrices
  \begin{equation*}
    \begin{array}{lclcl}
      A_1&=&\begin{bmatrix}
       3.7   &   -6.5   &   -3.6     & -3.1    &   3.8   \\
   -2.1      & 1.6     &   0.3      & 1.8      &-1.8   \\
    1.3     & -1.9     & - 0.7     & -1.3      & 1.8   \\
    3.3     & -10      & -6.8     & -2.7     &  4.8   \\
   -1.9     & -3.2     & -3.9    &   2.1     & -0.9
      \end{bmatrix},\ B_1&=&\begin{bmatrix}
         0.1     &   0.1   \\
     0.8     &   0.6   \\
     0.3      &  0.8   \\
     0.9      &  0.7   \\
     0.8      &  0.9
      \end{bmatrix},
    \end{array}
  \end{equation*}
  \begin{equation*}
    \begin{array}{lclcl}
          A_2&=&\begin{bmatrix}
         0.7   &   - 0.7    &   1.7   &    1.3     & -0.6   \\
     2.1    &    0.5     & -0.3     & -0.6      & 1.6   \\
    -0.4     &  2.7     & -4.3     & -3.9      &  0.2   \\
    1.4     & -2.6    &   4.4      & 4      &   0.7   \\
   -0.8     &  1.2   &   -2      & -1.3    &    0.7
      \end{bmatrix},\ B_2&=&\begin{bmatrix}
         0.7      &  0.9   \\
     0.6       & 0.2   \\
     0.2     &   0.9   \\
       0       & 0.2   \\
     0.6          &   0
      \end{bmatrix}.
    \end{array}
  \end{equation*}
  This system turns out to be non-stabilizable under arbitrary switching when using the conditions of Corollary \ref{th:daafouz} as synthesis conditions. The stabilization problem, however, becomes solvable for $\tau=2$ using Theorem \ref{th:stabz}.
%
\end{example}

\section{$\ell_2$-gain computation under minimum dwell-time constraint}\label{sec:ext}

We extend here the proposed framework to the computation of an upper-bound on the $\ell_2$-gain of discrete-time switched systems under a mimum dwell-time constraint. To this aim, let us consider  the following discrete-time switched system
\begin{equation}\label{eq:l2syst}
  \begin{array}{rcl}
    x(t+1)&=&A_{\sigma(t)}x(t)+E_{\sigma(t)}w_{\sigma(t)}(t)\\
    z_{\sigma(t)}(t)&=&C_{\sigma(t)}x(t)+F_{\sigma(t)}w_{\sigma(t)}(t)
  \end{array}
\end{equation}
where $w_i\in\mathbb{R}^{p_i}$ and $z_i\in\mathbb{R}^{q_i}$ are the exogenous input and the controlled output of mode $i$, respectively. Since the dimension of the input and output signals vary over time, we define the $\ell_2$-gain of the system \eqref{eq:l2syst} under minimum dwell-time $\tau$ to be the smallest $\gamma>0$ verifying
\begin{equation}
  \sum_{q\in\mathbb{N}}\sum_{k=1}^{\tau_q-1}||z_{\sigma(\phi_q)}(\phi_q^k)||_2^2\le\gamma^2 \sum_{q\in\mathbb{N}}\sum_{k=1}^{\tau_q-1}||w_{\sigma(\phi_q)}(\phi_q^k)||_2^2
\end{equation}
where $\phi_q^k:=\phi_q+k$, for all $\tau_q\ge\tau$ and all $\phi_q\in\{1,\ldots,N\}$, $q\in\mathbb{N}$ along the trajectories of the system \eqref{eq:l2syst} with zero initial conditions.

We have the following result:
\begin{theorem}\label{th:affine_L2}
 The following statements are equivalent:
 \begin{enumerate}
   \item There exist matrices $P_i\in\mathbb{S}_{\succ0}^n$, $i=1,\ldots,N$ and a scalar $\gamma>0$ such that the LMIs
  \begin{equation}\label{eq:l2_a1}
  \begin{bmatrix}
    A_i^\prime P_iA_i-P_i+C_i^\prime C_i &  A_i^\prime P_iE_i+C_i^\prime F_i\\
    \star & E_i^\prime P_iE_i+F_i^\prime F_i-\gamma^2I_{p_i}
  \end{bmatrix}\prec0
  \end{equation}
  and
  \begin{equation}\label{eq:l2_a2}
  \beth_{ij}:=\begin{bmatrix}
      -P_j & 0\\
      \star & -\gamma^2I_{p_i}
    \end{bmatrix}+\gimel_i\begin{bmatrix}
      P_i & 0\\
      0 & I_{q_i}
    \end{bmatrix}\gimel_i^\prime\prec0
  \end{equation}
  with
\begin{equation*}
        \gimel_i:=\begin{bmatrix}
      A_i^{\prime\tau} &  \mathcal{C}_i(\tau)^\prime\\
      E_i^{\prime\tau} &  \mathcal{F}_i(\tau)^\prime
    \end{bmatrix}
\end{equation*}
  hold for all $i,j=1,\ldots,N$, $i\ne j$ where $\mathcal{C}_i(1)=C_i$, $\mathcal{E}_i(1)=E_i$  and $\mathcal{F}_i(1)=F_i$, and, when $\tau\ge2$,
  \begin{equation}
        \mathcal{C}_i(\tau)=\begin{bmatrix}
          C_i\\
          C_iA_i\\
          \vdots\\
          C_iA_i^{\tau-1}
        \end{bmatrix},\ \mathcal{E}_i(\tau)=\begin{bmatrix}
          (A_i^{\tau-1}E_i)^\prime\\
           (A_i^{\tau-2}E_i)^\prime\\
          \vdots\\
          E_i^\prime
        \end{bmatrix}^\prime
  \end{equation}
  and $\mathcal{F}_i(\tau)$ is the lower triangular Toeplitz matrix of Markov parameters $h_i(k)=C_iA_i^{k-1}E_i$, $k\ge1$, $h_i(0)=F_i$ up to order $k=\tau-1$. That is, the first column of $\mathcal{F}_i(\tau)$ is given by $\col(F_i,C_iE_i, C_iA_iE_i,\ldots,C_iA_i^{\tau-2}E_i)$.
   \item There exist matrix sequences $R_i:\{0,\ldots,\tau\}\to\mathbb{S}^n$, $R_i(0)\succ0$, $i=1,\ldots,N$, and scalars $\eps>0$, $\gamma>0$ such that the LMIs
   \begin{equation}\label{eq:LMIl2_1}
 \Xi^i(\tau):=\begin{bmatrix}
       \Xi_{11}^i(\tau,\tau)  &   \Xi_{12}^i(\tau)\\
      \star &  \Xi_{22}^i(\tau)
   \end{bmatrix}\prec0
   \end{equation}
    \begin{equation}\label{eq:LMIl2_2}
    \Xi^i(k+1,k):=\begin{bmatrix}
       \Xi_{11}^i(k+1,k)  &   \Xi_{12}^i(k+1)\\
      \star &  \Xi_{22}^i(k+1)
    \end{bmatrix}\preceq0
  \end{equation}
  and
  \begin{equation}\label{eq:LMIl2_3}
    R_i(0)-R_j(\tau)+\eps I\preceq0
  \end{equation}
  hold for all $i,j=1,\ldots,N$, $i\ne j$ and $k=0,\ldots,\tau-1$ where
  \begin{equation*}
    \begin{array}{rcl}
      \Xi_{11}^i(\theta_1,\theta_2)&=&A_i^\prime R_i(\theta_1)A_i-R_i(\theta_2)+C_i^\prime C_i\\
      \Xi_{12}^i(\theta)&=&A_i^\prime R_i(\theta)E_i+C_i^\prime F_i\\
      \Xi_{22}^i(\theta)&=&E_i^\prime R_i(\theta)E_i+F_i^\prime F_i-\gamma^2I_{p_i}.
    \end{array}
  \end{equation*}
     \item There exist matrix sequences $S_i:\{0,\ldots,\tau\}\to\mathbb{S}^n$, $S_i(\tau)\succ0$, $i=1,\ldots,N$, and scalars $\eps>0$, $\gamma>0$ such that the LMIs
   \begin{equation}
   \begin{bmatrix}
     \Gamma_{11}^i(\tau,\tau) & E_i & A_iS_i(\tau)C_i^\prime\\
     \star & -\gamma^2I_{p_i} & F_i^\prime\\
     \star  & \star & -I_{q_i}+C_iS_i(\tau)C_i^\prime
   \end{bmatrix}\prec0
   \end{equation}
    \begin{equation}
    \begin{bmatrix}
      \Gamma_{11}^i(k,k+1) &  E_i & A_iS_i(k)C_i^\prime\\
     \star & -\gamma^2I_{p_i} & F_i^\prime\\
     \star  & \star & -I_{q_i}+C_iS_i(k)C_i^\prime
    \end{bmatrix}\preceq0
  \end{equation}
  and
  \begin{equation}
    S_j(\tau)-S_i(0)+\eps I\preceq0
  \end{equation}
  hold for all $i,j=1,\ldots,N$, $i\ne j$ and $k=0,\ldots,\tau-1$ where
  \begin{equation*}
      \Gamma_{11}^i(\theta_1,\theta_2)=A_iS_i(\theta_1)A_i^{\prime}-S_i(\theta_2).
  \end{equation*}
   \end{enumerate}
   Moreover, when one of the above equivalent statements holds, then the switched system \eqref{eq:l2syst} is asymptotically stable for all switching-time sequences $\{\phi_s\}_{s\in\mathbb{N}}$ satisfying $\tau_s\ge \tau$ and the $\ell_2$-gain of the transfer $w\mapsto z$ is less than $\gamma$.
  \end{theorem}
  \begin{proof}
  The equivalence between statements b) and c) follows from Schur complements.

  \textbf{Proof of (b) $\Rightarrow$ (a):} First, choosing $R_i(\tau)=P_i$ shows that the conditions \eqref{eq:l2_a1} and \eqref{eq:LMIl2_1} are identical.  Let $\Lambda^i(k)$ be defined as
  \begin{equation}
    \Lambda^i(k):= \begin{bmatrix}
      x(\phi_q+k)\\
      w(\phi_q+k)
    \end{bmatrix}^\prime \Xi^{\sigma(\phi_q)}(k,k+1)\begin{bmatrix}
      x(\phi_q+k)\\
      w(\phi_q+k)
    \end{bmatrix}
  \end{equation}
  Let us define $V(x,i,k):=x^\prime R_i(k)x$, then from \eqref{eq:LMIl2_2}, we have that
  \begin{equation}
  \begin{array}{rcl}
    \Lambda^{\sigma(\phi_q)}(k)&=&V(x(\phi_q+k+1),{\sigma(\phi_q)},k+1)\\
   &&-V(x(\phi_q+k),{\sigma(\phi_q)},k)\\
    &&-\gamma^2w(\phi_q+k)^\prime w(\phi_q+k)\\
    &&+z(\phi_q+k)^\prime z(\phi_q+k)\le0.
  \end{array}
  \end{equation}
  Summing over $k$ yields
  \begin{equation}\label{eq:dlksmdskdl1}
    \begin{array}{rcl}
    \sum_{k=0}^{\tau-1}\Lambda^{\sigma(\phi_q)}(k)&=&V(x(\phi_q+\tau),{\sigma(\phi_q)},\tau)\\
    &&-V(x(\phi_q),{\sigma(\phi_q)},0)\\
    &&+\sum_{k=0}^{\tau-1}||z(\phi_q+k)||_2^2\\
    &&-\gamma^2\sum_{k=0}^{\tau-1}||w(\phi_q+k)||_2^2\le0
  \end{array}
  \end{equation}
  Considering now \eqref{eq:LMIl2_3}, we obtain that
    \begin{equation}\label{eq:dlksmdskdl2}
    \begin{array}{l}
  V(x(\phi_q+\tau),{\sigma(\phi_q)},\tau)-V(x(\phi_q),{\sigma(\phi_{q-1})},\tau)\\
    +\sum_{k=0}^{\tau-1}||z(\phi_q+k)||_2^2\\
    -\sum_{k=0}^{\tau-1}\gamma^2||w(\phi_q+k)||_2^2\le-\eps ||x(\phi_q)||_2^2.
  \end{array}
  \end{equation}
  This condition is equivalent to saying that
  \begin{equation}
   \zeta_q(\tau)^\prime \beth_{ij} \zeta_q(\tau)\prec0
  \end{equation}
  where $\zeta_q(\tau)=\col(x(\phi_q), w(\phi_q),\cdots, w(\phi_q+\tau-1))$, $i=\sigma(\phi_q)$ and $j=\sigma(\phi_{q-1})$. The result follows.

    \textbf{Proof of (a) $\Rightarrow$ (b):} The proof follows the same line as for standard stability. We define recursively the lifted LMI \eqref{eq:l2_a2} by first pre- and post-multiplying by  $\zeta_q(\tau)^\prime$ and  $\zeta_q(\tau)$. Then, we obtain exactly \eqref{eq:dlksmdskdl2} to which we apply the inverse procedure to get \eqref{eq:dlksmdskdl1} and \eqref{eq:LMIl2_3}, where we use the identity $P_i=R_i(\tau)$. Then, we introduce auxiliary matrices $R_i(k)$ such that
    \begin{equation}
    \begin{array}{l}
            V(x(\phi_q+\tau),{\sigma(\phi_q)},\tau)-V(x(\phi_q),{\sigma(\phi_{q-1})},0)\\
            =\sum_{k=0}^{\tau-1}\left[V(x(\phi_q+k+1),{\sigma(\phi_q)},k+1)\right.\\
            \qquad-\left.V(x(\phi_q+k),{\sigma(\phi_{q-1})},k)\right].
    \end{array}
    \end{equation}
    We know that these matrices exist for the proof of Theorem \ref{th:affine}. By gathering the terms in $k$, i.e. $R_i(k)$, $x(\phi_q+k)$, $w(\phi_q+k)$ and $z(\phi_q+k)$, we get that
    \begin{equation}
      \begin{array}{lcl}
         \zeta_q(\tau)^\prime \beth_{ij} \zeta_q(\tau)&=&\sum_{k=0}^{\tau-1}\Lambda^{\sigma(\phi_q)}(k)\preceq0
      \end{array}
    \end{equation}
    where $i=\sigma(\phi_q)$ and $j=\sigma(\phi_{q-1})$. Using finally the fact that each one of the outer vectors in the $\Lambda^{\sigma(\phi_q)}(k)$'s is independent of the others, this implies that all the quadratic forms $\Lambda^{\sigma(\phi_q)}(k)$'s are semidefinite, and the conclusion follows.

  \textbf{Proof of (b) $\Rightarrow$ $\ell_2$-gain is smaller than $\gamma>0$:}
  The sum \eqref{eq:dlksmdskdl2} can completed up to $k=\tau_q-1$ using \eqref{eq:LMIl2_2} and we get that
  \begin{equation}\label{eq:kdlsqL2}
    \begin{array}{l}
      V(x(\phi_q+\tau_q),{\sigma(\phi_q)},\tau)-V(x(\phi_q),{\sigma(\phi_{q-1})},\tau)\\
    +\sum_{k=0}^{\tau_q-1}||z(\phi_q+k)||_2^2\\
    -\sum_{k=0}^{\tau_q-1}\gamma^2||w(\phi_q+k)||_2^2\le-\eps ||x(\phi_q)||_2^2.
    \end{array}
  \end{equation}
  Since the LMIs \eqref{eq:LMIl2_1}-\eqref{eq:LMIl2_2}-\eqref{eq:LMIl2_3} implies stability under minimum dwell-time $\tau$ and that $\phi_{q+1}=\phi_q+\tau_q$, $\phi_0=0$, we have that $V(x(\phi_{q+1}),\sigma(\phi_q),\tau)\to0$ as $q\to\infty$. Summing then the inequality \eqref{eq:kdlsqL2} over $q$ yields that
  \begin{equation}
    \begin{array}{l}
    -V(x(0),\sigma(\phi_{-1}),\tau)+\sum_{k=0}^\infty ||z(k)||_2^2\\
    \qquad-\sum_{k=0}^{\infty}\gamma^2||w(k)||_2^2\le-\eps\sum_{q=0}^\infty||x(\phi_q)||_2^2
    \end{array}
  \end{equation}
  where $\sigma(\phi_{-1})\in\{1,\ldots,N\}$ is arbitrary. This finally gives that
  \begin{equation}
    \sum_{k=0}^\infty ||z(k)||_2^2<\sum_{k=0}^{\infty}\gamma^2||w(k)||_2^2+V(x(\phi_0),\sigma(\phi_{-1}),\tau)
  \end{equation}
from which the $\ell_2$-gain conclusion follows.
  \end{proof}

It seems important to stress that the proposed method is simpler than the one reported in \cite{Colaneri:11}. The dwell-time idea is the same, that is, taken from the paper \cite{Geromel:06d}. The approach to compute the $\ell_2$-gain, however, is more complex in \cite{Colaneri:11} since it relies on some approximate matrix computations. The proposed approach circumvents this and solves the problem in a more direct manner. Note, moreover, that the current approach is clearly applicable in a stabilization and uncertain contexts, unlike the method of \cite{Colaneri:11} which involves matrix powers. 

\begin{example}\label{ex:l2}
  Let us consider the example considered in \citep{Colaneri:11}
  \begin{equation}
    \begin{array}{lcl}
      \begin{bmatrix}
        A_1 & \vline & E_1\\
        \hline
        C_1 & \vline & F_1
      \end{bmatrix}&=&\begin{bmatrix}
            0 & 0.25 & 0 & \vline & 1\\
            1 & 0 & 0 & \vline & 0\\
            0 & 1 & 0 & \vline & 0\\
            \hline
            0 & 0 & 0.7491 & \vline & 0
      \end{bmatrix},\\
      \begin{bmatrix}
        A_2 & \vline & E_2\\
        \hline
        C_2 & \vline & F_2
      \end{bmatrix}&=&\begin{bmatrix}
            -2 & -1.5625 & -0.4063 & \vline & 1\\
            1 & 0 & 0 & \vline & 0\\
            0 & 1 & 0 & \vline & 0\\
            \hline
            0.0964 & 0.0964 & 0.0964 & \vline & 0
      \end{bmatrix},\\
      \begin{bmatrix}
        A_3 & \vline & E_3\\
        \hline
        C_3 & \vline & F_3
      \end{bmatrix}&=&\begin{bmatrix}
        1 & -0.5625 & 0.1563 & \vline & 1\\
        1 & 0 & 0 & \vline & 0\\
        0 & 1 & 0 & \vline & 0\\
        \hline
        0.2031 & 0.0444 & 0.1174 & \vline & 0.1015
      \end{bmatrix}.
    \end{array}
  \end{equation}
 Using the results of the paper, we find that the system is not stable for a dwell-time $\tau<5$. Exactness of this result is confirmed by the fact that the product $A_2^4A_3^4$ has an eigenvalue outside the unit disc. We compute then the $\ell_2$-gain of the above system using the statement b) of Theorem \ref{th:affine_L2} for different values for $\tau$ ranging from 5 to 40. We obtain the curve depicted in Figure \ref{fig:l2} which is very similar to the one obtained in \cite{Colaneri:11}.
 \begin{figure}[h]
 \centering
   \includegraphics[width=0.5\textwidth]{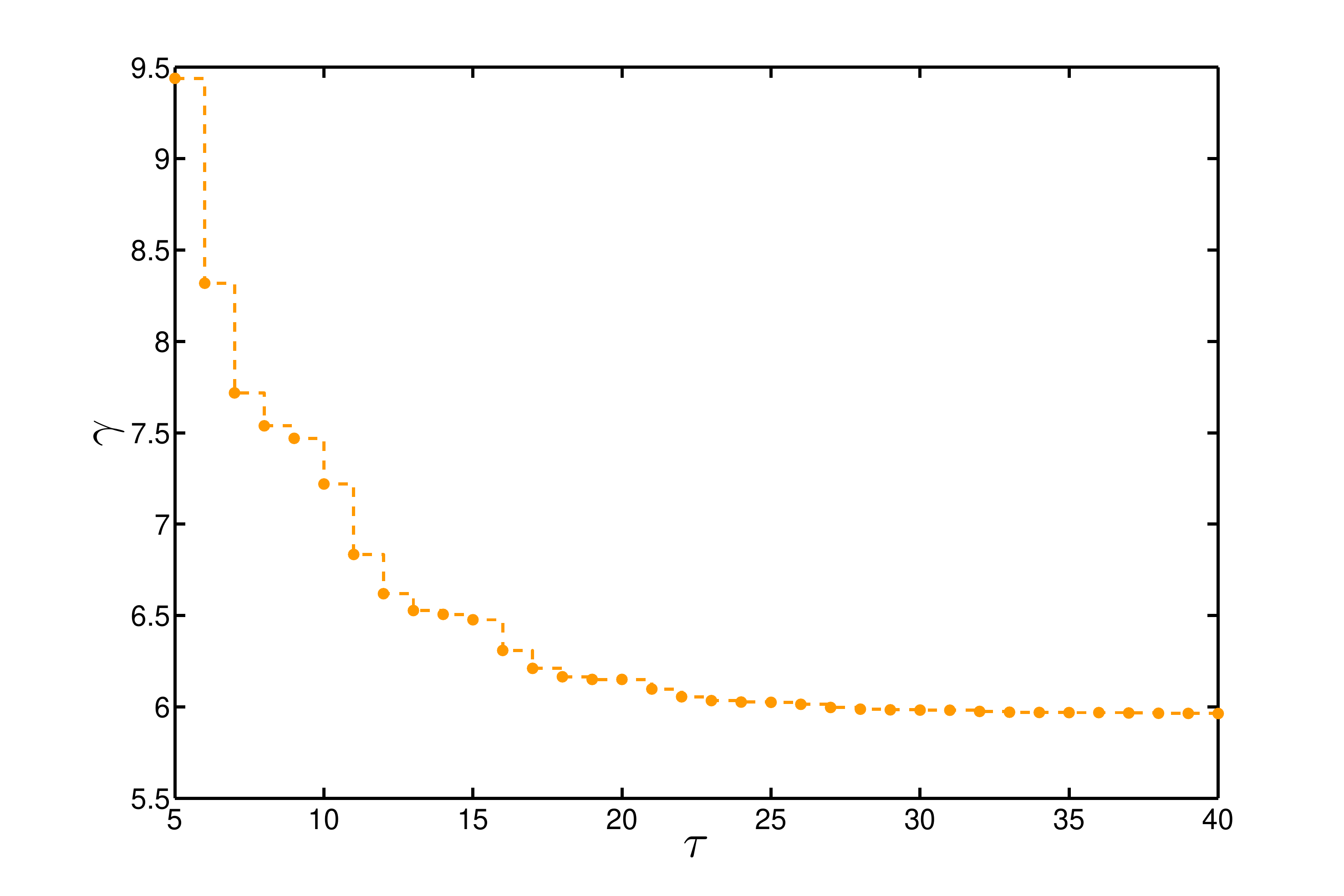}
   \caption{$\ell_2$-gain vs.  minimal dwell-time $\tau$ for the system of Example \ref{ex:l2}.}\label{fig:l2}
 \end{figure}
\end{example}

\begin{example}\label{ex:l2stabz}
Let us consider here the following open-loop system
\begin{equation}\label{eq:oll2}
  \begin{array}{rcl}
    x(t+1)&=&A_{\sigma(t)}x(t)+B_{\sigma(t)}u_{\sigma(t)}(t)+E_{\sigma(t)}w_{\sigma(t)}(t)\\
    z_{\sigma(t)}(t)&=&C_{\sigma(t)}x(t)+D_{\sigma(t)}u_{\sigma(t)}(t)+F_{\sigma(t)}w_{\sigma(t)}(t)
  \end{array}
\end{equation}
where, as before, $u$ is the control input, $w$ is the exogenous input and $z$ is the controlled output. Again, the dimension of these signals may vary over time. The goal is to design a controller of the form \eqref{eq:impsf} such that the closed-loop system is stable with minimum dwell-time $\tau$ and has suboptimal minimal $\ell_2$-gain. Let us choose then the matrices $F_1=F_2=D_1=D_2=0$. $C_1=C_2=\begin{bmatrix}
  1 & 0
\end{bmatrix}$, $B_1=B_2=\begin{bmatrix}
  1 & 0
\end{bmatrix}^\prime$,
\begin{equation*}
  A_1=\begin{bmatrix}
    1 & 2\\ 3 & 1
  \end{bmatrix}\ \textnormal{and}\   A_2=\begin{bmatrix}
    1 & 2\\ -8 & 1
  \end{bmatrix}.
\end{equation*}
We then design controllers of the form \eqref{eq:impsf} that suboptimally minimize the $\ell_2$-gain $\gamma$ of the closed-loop system with respect to some minimum dwell-time constraint. This is done by  suitably adapting the conditions of Theorem \ref{th:stabz}. The results are depicted in Figure \ref{fig:l2stabz} where we have plotted the suboptimal $\gamma$'s with respect to $\tau$. This system is not stabilizable for $\tau=1$, so the plot starts at $\tau=2$. We can observe that by increasing the minimum dwell-time, the $\ell_2$-gain can be made smaller.
\end{example}

 \begin{figure}[h]
 \centering
   \includegraphics[width=0.5\textwidth]{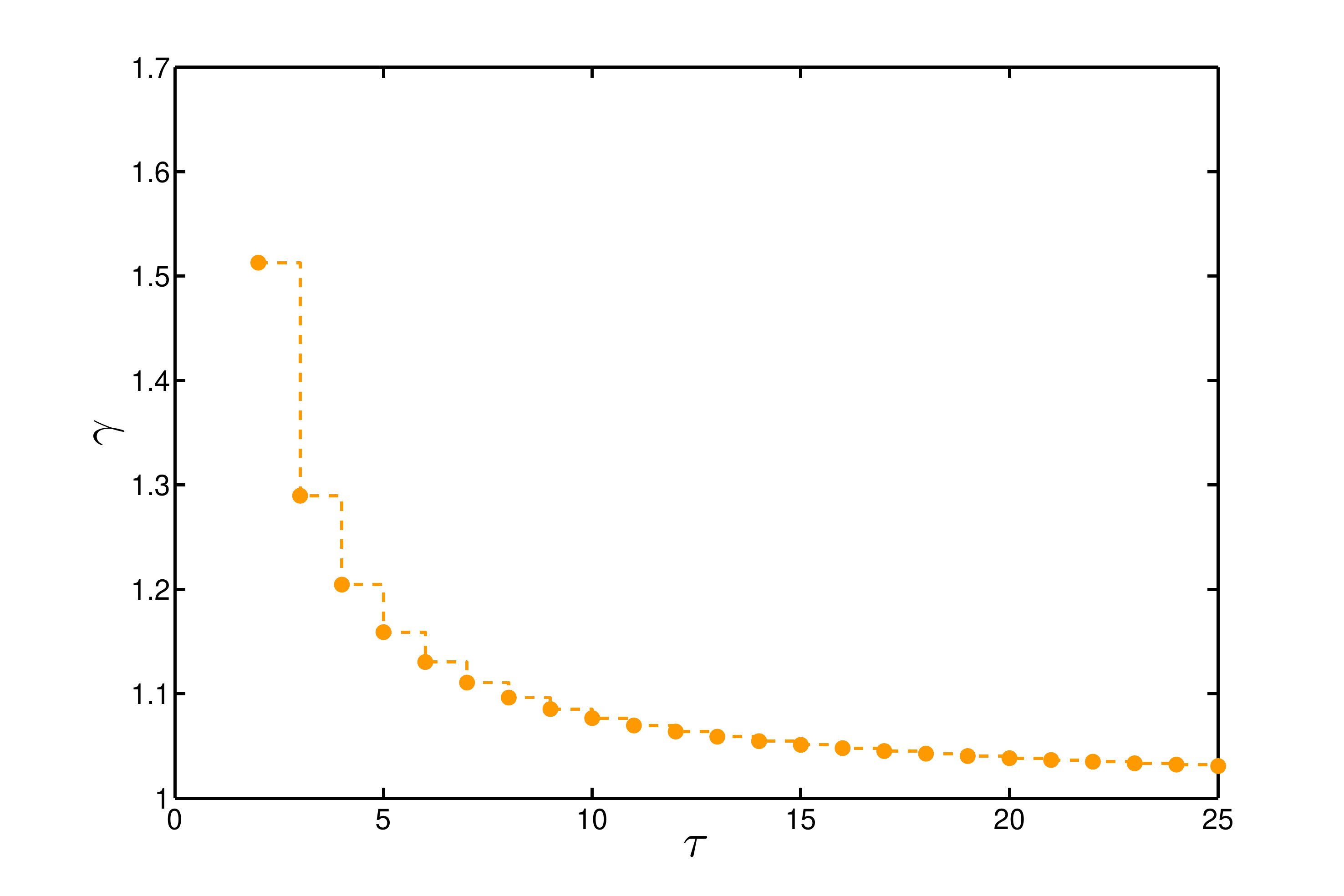}
   \caption{Suboptimal $\ell_2$-gain $\gamma$ vs.  minimal dwell-time $\tau$ for the system of Example \ref{ex:l2stabz}.}\label{fig:l2stabz}
 \end{figure}

%

\section{Conclusion}

New conditions for the analysis of discrete-time switched systems have been proposed. They have been shown to be equivalent to the minimum dwell-time conditions of \citep{Geromel:06d}. Due to their convex structure, the conditions can be extended to the case of uncertain switched with time-varying uncertainties and to control design using a non-restrictive class of time-varying controllers. Following the same ideas, a convex criterion for computing the $\ell_2$-gain of a discrete-time switched system under a minimum dwell-time constraint has been provided. The interest of the approach also lies on its possible generalization to nonlinear systems, LPV systems and to systems with time-varying delays \citep{Hetel:08b}. Other performance criteria such as the $\ell_2$-$\ell_\infty$-gain, or the $\ell_1$-gain can be easily considered as well.

Practical applications can be found in the analysis and control of networked control systems; see e.g. \citep{Donkers:11,Deaecto:13}. The latter paper addresses the problem of scheduling between different discrete-time systems obtained from the discretization of a continuous-time system using different sampling periods. The authors then develop a suitable scheduling strategy using the theory of discrete-time switched systems. The authors, however, mention a difficulty of their method which is the co-design of controllers (e.g. state-feedback)  and the scheduling strategy. The approach of this paper can be adapted to solve this problem. This is left for future research.

\bibliographystyle{plainnat}
\bibliography{../../../../Lastbib/global,../../../../Lastbib/briat}

\end{document}